\documentclass[10pt,reqno]{amsart}
\usepackage{amsmath,amssymb,amsfonts,amsthm,mathrsfs,enumitem}
\allowdisplaybreaks

\newtheorem{thm}{Theorem}[section]

\newtheorem{lem}[thm]{Lemma}

\newtheorem{con}[thm]{Conjecture}
\theoremstyle{definition}

\theoremstyle{remark}
\newtheorem{rem}[thm]{Remark}

\numberwithin{equation}{section}

\newcommand{\F}{\mathbb F}
\newcommand{\Tr}{\operatorname{Tr}}
\newcommand{\im}{\operatorname{im}}
\newcommand{\one}{\mathbf 1}
\newcommand{\gauss}[2]{\genfrac{[}{]}{0pt}{}{#1}{#2}_{\!2}}

\begin{document}

\title[Counting zero-sum subspaces]
{Counting zero-sum subspaces for the multiplicative inverse function}
\author{Kaimin Cheng}
\address{School of Mathematical Sciences, China West Normal University, Nanchong 637002, P. R. China}
\email{ckm20@126.com}
\date{}

\begin{abstract}
Let $q=2^n$ with $n\ge6$, and let $\F_q$ be the finite field with $q$
elements. For a subspace $E$ of $\F_q$, define
\[
S(E)=\sum_{x\in E\setminus\{0\}}x^{-1}.
\]
Let $N_{n,k}$ denote the number of $k$-dimensional $\F_2$-subspaces $E$ for which
$S(E)=0$. We prove that, if $k\ge3$ and $n\ge2k+1$, then
\[
\left|\frac{2^nN_{n,k}}{\gauss nk}-1\right|
 <2^{r^2+r}2^{n(1-r/2)},
\]
where
$r=\lfloor\frac{k-1}{2}\rfloor$, and $\gauss nk$ is the Gaussian binomial coefficient. In particular,
\[
N_{n,k}\sim2^{-n}\gauss nk\qquad (n\to\infty)
\]
uniformly for $n/3<k\le(n-1)/2$. Combined with the known
low-dimensional cases, the symmetry between dimensions $k$ and $n-k$, and
the elementary middle-dimensional subfield construction, this proves a
conjecture of Carlet: for every $3\le k\le n-3$, there exists a
$k$-dimensional $\F_2$-subspace $E$ such that $S(E)=0$.
\end{abstract}

\maketitle

\section{Introduction}\label{sec:intro}

Let $q=2^n$ with $n\ge2$, and let $\F_q$ be the finite field with $q$ elements. Throughout the paper, every subspace of $\F_q$ is understood
to be an $\F_2$-subspace, and all dimensions are taken over $\F_2$. The
multiplicative inverse function
\[
 F_{\rm inv}(x)=x^{q-2}\quad (x\in\F_q),
\]
where $F_{\rm inv}(0)=0$, is a basic example of a vectorial Boolean function.
Its second-order behaviour is classical: it is almost perfect nonlinear
(APN) exactly when $n$ is odd; see, for example, \cite{Nyberg}. APN
functions play an important role in the study of S-boxes and differential
cryptanalysis.

Carlet \cite{CarletAPN} introduced $k$th-order sum-freedom as a higher-order
analogue of the sum characterization of APN functions. An $(n,m)$-function
$F:\F_2^n\to\F_2^m$ is called $k$th-order sum-free if the sum of its values
over every $k$-dimensional affine subspace is nonzero. The notion is also
related to higher-order and integral cryptanalysis, where sums over affine
subspaces naturally occur; see \cite{CarletAPN,EbelingHouRydellZhao} and the
references therein. Further $q$-ary and coding-theoretic aspects have recently
been studied in \cite{HouZhaoFurther}. The inverse function is therefore a
natural and particularly important example for the study of sum-freedom.

Carlet \cite{CarletAPN,Carlet} observed that affine subspaces which do not
contain $0$ present no obstruction. In fact, if $V$ is a vector subspace of
$\F_q$ and $a\notin V$, then
\[
 \sum_{x\in a+V}x^{-1}
 =\frac{L_V'(a)}{L_V(a)}
 =\frac{b_0}{L_V(a)}\ne0,
\]
where
\[
 L_V(X)=\prod_{v\in V}(X+v)
\]
is the subspace polynomial of $V$, and $b_0\ne0$ is its coefficient of $X$.
Thus only vector subspaces need to be considered. For a subspace
$E$ of $\F_q$, let
\[
 S(E)=\sum_{x\in E\setminus\{0\}}x^{-1},
\]
and, for $0\le k\le n$, let
\[
 N_{n,k}=\#\{\text{subspaces }E\text{ of }\F_q:\dim E=k,\ S(E)=0\}.
\]
Carlet \cite{Carlet} raised the following conjecture.

\begin{con}\label{con:carlet}
For every $n\ge6$ and every $3\le k\le n-3$, one has $N_{n,k}>0$.
\end{con}

Several recent papers have made substantial progress on this problem.
Carlet \cite{Carlet} introduced the formulation in terms of subspace
polynomials, proved the symmetry between dimensions $k$ and $n-k$, treated
several initial families, and identified the importance of the next-to-leading
coefficient of a subspace polynomial. Carlet and Hou \cite{CarletHou} used
Frobenius-invariant subspaces and point-counting methods. In particular, they
proved the conjecture for even $n$ and obtained further results for relatively
small dimensions. Ebeling, Hou, Rydell and Zhao
\cite{EbelingHouRydellZhao} found additional arithmetic families, including
cases in which $n$ has small prime divisors, and developed the general theory
of sum-free functions further.

Carlet and Hou \cite{CarletHou} proved the absolute irreducibility of the
symmetric polynomial $F_k$ for $k\ge3$. Hou and Zhao \cite{HouZhaoFFA}
proved the corresponding absolute irreducibility result for a second symmetric
polynomial $\Theta_k$. Their point-counting estimates, together with
low-dimensional cases and a dimension-addition construction, yield the
conjecture, for $n\ge8$, in the ranges
\[
 3\le k\le \left\lceil\frac n3\right\rceil+2
 \quad\text{or}\quad
 \left\lfloor\frac{2n}{3}\right\rfloor-2\le k\le n-3.
\]
They also settled every fixed dimension $3\le k\le12$ when $n\ge k+3$.
In a subsequent paper \cite{HouZhaoDCC}, they proved the conjecture for every
nonprime extension degree $n$.

Thus the remaining cases are confined to prime extension degrees. The existing
methods, however, do not give a uniform treatment of all dimensions in this
situation. In particular, their quantitative limitations become pronounced
when $k$ grows proportionally with $n$. The nontrivial intermediate subfields used in many of the known
constructions are then unavailable, and the point-counting
approach faces a further difficulty: the two hypersurfaces involved have degrees
\[
 \deg F_k=2^k-2,\qquad \deg\Theta_k=2^{k-1}.
\]
When $k$ is proportional to $n$, these degrees are exponential in $n$, and the
usual point-counting error terms do not provide uniform information in the
central range. What is needed is therefore an argument that remains effective
when the dimension grows with the extension degree.

Our argument takes a different route and does not rely on the primality of
$n$. In fact, it gives a uniform quantitative estimate for arbitrary $n$ in
the range $n\ge2k+1$, while avoiding point counting on the high-degree
hypersurfaces above. For a $d$-dimensional subspace $H$, let $L_H$ be its
subspace polynomial and let $\tau(H)$ denote its next-to-leading coefficient.
Rather than keeping track only of the subspaces for which $\tau(H)=0$, we consider the entire distribution of $\tau(H)$ among $d$-dimensional subspaces containing $1$. The count at $\tau(H)=0$ alone does not satisfy a closed recursion under the passage from dimension $d$ to dimension $d+1$, whereas the full distribution does. The map $x\mapsto x^2+x$ then yields an exact recursion relating these distributions in successive dimensions.

After subtracting the uniform distribution, the recursion becomes a linear
operator. Additive character orthogonality gives an explicit formula for its
Euclidean norm. In each step, all but one direction are contracted. The
exceptional direction is mapped to a direction which is almost orthogonal to
the exceptional direction for the next step. Thus two successive steps give a
uniform contraction. This makes it possible to handle dimensions that grow
proportionally with $n$.

We use the standard Gaussian binomial coefficient
\[
 \gauss nk=\gauss{n}{k}
 =\prod_{i=0}^{k-1}\frac{2^n-2^i}{2^k-2^i},
\]
which precisely counts the $k$-dimensional
subspaces of $\F_q$. For brevity, we write
\[
G_{n,k}=\gauss nk.
\]
The following estimate is the quantitative form of our result.

\begin{thm}\label{thm:quant}
Let $k\ge3$, let $n\ge2k+1$, and put $r=\lfloor\frac{k-1}{2}\rfloor$. Then
\begin{equation}\label{eq:mainbound}
 \left|\frac{2^nN_{n,k}}{G_{n,k}}-1\right|
<2^{r^2+r}2^{n(1-r/2)}.
\end{equation}
In particular, if $k\ge10$, then
\begin{equation}\label{eq:positive-half}
 N_{n,k}>\frac{G_{n,k}}{2^{n+1}}.
\end{equation}
Moreover, for all sufficiently large $n$, we have
\begin{equation}\label{eq:uniformcentral}
 \sup_{n/3<k\le(n-1)/2}
 \left|\frac{2^nN_{n,k}}{G_{n,k}}-1\right|
 <2^{-n^2/18+4n/3}.
\end{equation}
\end{thm}

The last estimate is considerably stronger than the existence statement. It
shows that, throughout the central interval, the proportion of zero-sum
$k$-spaces is asymptotic to $2^{-n}$, uniformly in $k$. In particular, the
endpoint $k=(n-1)/2$ is included. This is also the range in which the degrees
of the hypersurfaces described above are largest in relation to the field
size.

\begin{thm}\label{thm:carlet}
Conjecture~\ref{con:carlet} holds. Equivalently, for every $n\ge6$ and every
$3\le k\le n-3$, the multiplicative inverse function on $\F_{2^n}$ is not
$k$th-order sum-free.
\end{thm}

The paper is organized as follows. Section~\ref{sec:prelim} recalls the
subspace-polynomial identities and the duality between the inverse sum and the
next-to-leading coefficient. In Sect.~\ref{sec:transfer} we derive the exact
recursion for the coefficient distribution among subspaces containing $1$.
Section~\ref{sec:fourier} studies the corresponding linear operator. We first
compute its one-step norm and then reduce the norm of two successive steps to
a two-dimensional calculation. Section~\ref{sec:proof} contains the proofs
of the two main theorems.

\section{Preliminaries on subspace polynomials}\label{sec:prelim}

All subspaces of $\F_q$ and their dimensions are taken over $\F_2$, and we put
$q=2^n$. We write $\Tr=\Tr_{\F_q/\F_2}$ for the absolute trace and
$\im L=\{L(x):x\in\F_q\}$ for the image of an $\F_2$-linear map $L$.
For a nonzero subspace $H$ of dimension $d$, its subspace polynomial is
\begin{equation}\label{eq:LH}
 L_H(X)=\prod_{h\in H}(X+h)=\sum_{i=0}^d b_iX^{2^i},
 \qquad b_d=1,\quad b_0\ne0.
\end{equation}
This is a monic linearized polynomial whose set of zeros is exactly $H$. We
write
\[
 \tau(H)=b_{d-1}
\]
for its next-to-leading coefficient. In particular, when $d=1$ one has
$b_1=1$ and $\tau(H)=b_0$.

Let us begin with two standard coefficient identities that relate the inverse sum
to the subspace polynomial.

\begin{lem}\label{lem:basic}
Let $H$ be as in \eqref{eq:LH}. Then
\begin{equation}\label{eq:Scoef}
 S(H)=\frac{b_1}{b_0}.
\end{equation}
If $U=\im L_H$, then $\dim U=n-d$, and
\begin{equation}\label{eq:SU}
 S(U)=\frac{b_1}{b_0^2}.
\end{equation}
The map $H\mapsto \operatorname{im} L_H$ is a bijection between the set of
$d$-dimensional subspaces of $\F_q$ and the set of
$(n-d)$-dimensional subspaces of $\F_q$. In particular,
\begin{equation}\label{eq:complement}
 N_{n,d}=N_{n,n-d}.
\end{equation}
\end{lem}

\begin{proof}
Note that the coefficient of $X$ in
$L_H(X)/X=\prod_{0\ne h\in H}(X+h)$ is $b_0S(H)$. By comparison with
\eqref{eq:LH}, one obtains \eqref{eq:Scoef}.

If $d=n$, then $H=\F_q$ and $U=\{0\}$. Since $n\ge2$, both $S(H)$ and
$S(U)$ are zero, and all assertions are immediate. We therefore assume
$1\le d<n$.

Let $U=\im L_H$. The composition $L_U\circ L_H$ is monic, has degree
$2^n$, and vanishes at every element of $\F_q$. Hence
\[
 L_U\circ L_H=X^q+X.
\]
Write $L_U(Y)=c_0Y+c_1Y^2+\cdots$. Comparison of the coefficients of $X$
and $X^2$ yields
\[
 c_0b_0=1,\qquad c_0b_1+c_1b_0^2=0.
\]
By \eqref{eq:Scoef}, applied to $U$, these two identities imply
\eqref{eq:SU}.

It remains to justify the bijectivity of the image map. Suppose that
$H_1$ and $H_2$ have the same image $U$, and put
$P=L_{H_1}-L_{H_2}$. Since $L_U$ is linearized,
\[
 L_U\circ P=L_U\circ L_{H_1}-L_U\circ L_{H_2}=0.
\]
If $P$ were nonzero, then $P(0)=0$ would imply that $P$ is nonconstant.
Since $L_U$ is also nonconstant, the degree formula for polynomial composition
would give
\[
 \deg(L_U\circ P)=\deg(L_U)\deg(P)\ge1,
\]
a contradiction. Hence $L_{H_1}=L_{H_2}$ and therefore $H_1=H_2$.
The image map is injective. Since the sets of $d$-dimensional and
$(n-d)$-dimensional subspaces of $\F_q$ have the same cardinality, it is
bijective. Finally, \eqref{eq:complement} follows from \eqref{eq:Scoef} and
\eqref{eq:SU}.
\end{proof}

We shall also use the coefficient which is dual to $S(H)$. Let
$\langle x,y\rangle=\Tr(xy)$ be the absolute trace pairing.

\begin{lem}\label{lem:tau}
The number of $d$-dimensional subspaces $H$ of $\F_q$ for which $\tau(H)=0$ is
$N_{n,d}$.
\end{lem}

\begin{proof}
If $d=n$, then $H=\F_q$, and for $n\ge2$ one has
$\tau(H)=S(H)=0$. Thus the assertion is immediate. We henceforth assume
$1\le d<n$.

Let $A=(\im L_H)^\perp$, where the orthogonal complement is taken with
respect to the trace pairing $\langle x,y\rangle$. The adjoint of $L_H$
relative to this pairing is the linearized map
\[
 L_H^*(X)=\sum_{i=0}^d b_i^{2^{n-i}}X^{2^{n-i}}
 \pmod{X^q-X}.
\]
Hence $\ker L_H^*=A$. For $x\in\F_q$, the condition $L_H^*(x)=0$ is
unchanged after raising to the $2^d$-th power. Using $x^{2^n}=x$ in
$\F_q$, we may reduce the resulting Frobenius exponents modulo $n$. This
gives, on $\F_q$, the same zero set as the linearized polynomial
\[
 \sum_{i=0}^d b_i^{2^{d-i}}X^{2^{d-i}}.
\]
Its term of highest degree is $b_0^{2^d}X^{2^d}$. Since $b_0\ne0$, division
by $b_0^{2^d}$ therefore produces a monic linearized polynomial of degree
$2^d$ whose zero set is $A$. As $\dim A=d$, this polynomial must be the
subspace polynomial of $A$. Thus
\begin{equation}\label{eq:adjoint}
 L_A(X)=b_0^{-2^d}\sum_{i=0}^d b_i^{2^{d-i}}X^{2^{d-i}}.
\end{equation}

By Lemma~\ref{lem:basic}, $H\mapsto\im L_H$ is a bijection from the
$d$-dimensional subspaces onto the $(n-d)$-dimensional subspaces. Since the
trace pairing is nondegenerate, orthogonal complementation is a bijection from
the latter family back to the $d$-dimensional subspaces. Hence
$H\mapsto A=(\im L_H)^\perp$ is a bijection on the set of $d$-dimensional
subspaces. From the coefficients of $X$ and $X^2$ in \eqref{eq:adjoint}, and
then Lemma~\ref{lem:basic}, one obtains
\[
 S(A)=\tau(H)^2.
\]
Thus $\tau(H)=0$ if and only if $S(A)=0$. The desired equality of the two
counts follows. This is the coefficient duality used in
\cite{Carlet,HouZhaoFFA}.
\end{proof}

For $a\in\F_q$, let $\delta_a$ denote the delta function at $a$, that is,
\[
 \delta_a(x)=
 \begin{cases}
 1,& x=a,\\
 0,& x\ne a.
 \end{cases}
\]
For a condition $\mathcal P$, let $\one_{\{\mathcal P\}}$ denote its indicator
function.

For $1\le d\le n$ and $b\in\F_q$, define
\[
 h_d(b)=\#\{\text{subspaces }H\text{ of }\F_q:\dim H=d,\ 1\in H,\ \tau(H)=b\},
 \qquad
 B_d=\gauss{n-1}{d-1},
\]
and put
\[
 p_d(b)=\frac{h_d(b)}{B_d}.
\]
The function $p_d$ is a probability distribution on $\F_q$. More
precisely, if $H$ is chosen uniformly from the $d$-dimensional subspaces that
contain $1$, then
\[
 p_d(b)=\mathbb P(\tau(H)=b),\qquad b\in\F_q.
\]
For $d=1$, the only such subspace is $\F_2=\{0,1\}$, and its subspace
polynomial is $X^2+X$. Thus $p_1=\delta_1$.

\begin{lem}\label{lem:pzero}
For every $1\le d\le n$, we have
\begin{equation}\label{eq:pzero}
 p_d(0)=\frac{N_{n,d}}{G_{n,d}}.
\end{equation}
Moreover, if
\[
 M_d(b)=\#\{\text{subspaces }H\text{ of }\F_q:\dim H=d,\ \tau(H)=b\}
 \quad\text{and}\quad s=2^d,
\]
then
\begin{equation}\label{eq:M}
 M_d(0)=\frac{q-1}{s-1}h_d(0),
 \quad
M_d(b)=\frac{B_d-h_d(0)}{s-1}\quad(b\ne0).
\end{equation}
\end{lem}

\begin{proof}
Note that for $c\ne0$,
\[
 L_{cH}(X)=c^{2^d}L_H(X/c),
 \qquad
 \tau(cH)=c^{2^{d-1}}\tau(H).
\]
Since $c\mapsto c^{2^{d-1}}$ permutes $\F_q^*$, the number of subspaces
with $\tau(H)=b$ is the same for every $b\ne0$. Scalar multiplication also
preserves the condition $\tau(H)=0$. We now count pairs $(H,v)$ with
$\tau(H)=0$ and $0\ne v\in H$. For each such $H$ there are $s-1$ choices
of $v$, whereas for each fixed $v\ne0$ there are $h_d(0)$ such subspaces
after scaling by $v^{-1}$. Hence
\[
 (s-1)M_d(0)=(q-1)h_d(0).
\]
This proves the first formula in \eqref{eq:M}. By Lemma~\ref{lem:tau},
$M_d(0)=N_{n,d}$, and the standard incidence count gives
\begin{equation}\label{eq:B_d}
 B_d=\frac{s-1}{q-1}G_{n,d}.
\end{equation}
Combining these two identities gives \eqref{eq:pzero}.

Finally, let $m$ denote the common value of $M_d(b)$ for $b\ne0$. Since
\[
 G_{n,d}-M_d(0)=(q-1)m,
\]
the first formula in \eqref{eq:M}, together with \eqref{eq:B_d}, implies
\[
 m=\frac{B_d-h_d(0)}{s-1}.
\]
This proves the second formula in \eqref{eq:M}.
\end{proof}

\section{A closed recursion for the coefficient distribution}\label{sec:transfer}

Let
\[
 \psi(x)=(-1)^{\Tr(x)},
 \qquad
 T=\ker\Tr.
\]
Consider the $\F_2$-linear map
\[
 \phi(x)=x^2+x.
\]
Its kernel is $\F_2$ and its image is $T$. If $H$ is a
$(d+1)$-dimensional subspace containing $1$, the kernel of the restriction
$\phi|_H$ is therefore $\F_2$, so $\phi(H)$ is a $d$-dimensional subspace of
$T$. Conversely, for every $d$-dimensional subspace $V$ of $T$, the inverse
image $\phi^{-1}(V)$ is a $(d+1)$-dimensional subspace containing $1$.
Thus $H\mapsto\phi(H)$ is a bijection between these two families.

If $H$ contains $1$ and $V=\phi(H)$, then
\[
 L_H=L_V\circ\phi.
\]
Indeed, the right-hand side is monic of degree $2^{d+1}$ and its zeros are
exactly the elements of $H$. It must therefore be $L_H$. By comparing the
two highest terms, one obtains
\begin{equation}\label{eq:tau-phi}
 \tau(H)=1+\tau(V).
\end{equation}

The next theorem is the main algebraic ingredient of the proof.

\begin{thm}\label{thm:transfer}
Let $1\le d<n$ and put $s=2^d$. Then
\begin{align}
 p_{d+1}(1)&=p_d(0),\label{eq:transfer-one}\\
 p_{d+1}(b)&=
 \frac{1-p_d(0)+(s-1)\displaystyle\sum_{v\in\F_q^*}
       \psi((b+1)/v)p_d(v)}{q-s},
 \qquad b\ne1.\label{eq:transfer-other}
\end{align}
Thus the coefficient distribution among subspaces containing $1$ is determined
recursively from $p_1=\delta_1$.
\end{thm}

It is important here to retain the entire distribution. For an individual
subspace, the value of $\tau(H)$ alone does not determine how many extensions
have a prescribed next coefficient. After averaging over all subspaces,
however, the passage from dimension $d$ to $d+1$ depends only on $p_d$ itself.

\begin{proof}
First, for a $d$-dimensional subspace $V$, additive character orthogonality on
$V$ gives
\[
 \one_{\{V\subseteq T\}}=\frac1s\sum_{v\in V}\psi(v).
\]
Indeed, the restriction of $\psi$ to $V$ is trivial exactly when
$V\subseteq T$; otherwise its sum over $V$ is zero.

Now fix $b\in\F_q$ and put $x=b+1$. By the bijection $H\mapsto\phi(H)$ and
\eqref{eq:tau-phi}, the number $h_{d+1}(b)$ is the number of $d$-dimensional
subspaces $V$ satisfying
\[
 V\subseteq T,\qquad \tau(V)=x.
\]
Therefore
\begin{align*}
s h_{d+1}(b)
=\sum_{\substack{\dim V=d\\ \tau(V)=x}}
   \sum_{v\in V}\psi(v)=M_d(x)+ \sum_{v\in\F_q^*}\psi(v)
   \#\{V:\dim V=d,\ \tau(V)=x,\ v\in V\}.
\end{align*}
Fix $v\ne0$. Multiplication by $v^{-1}$ gives a bijection onto the
$d$-dimensional subspaces $W$ containing $1$ and satisfying
\[
\tau(W)=v^{-2^{d-1}}x.
\]
Hence the number of such spaces is
$h_d(x/v^{2^{d-1}})$. Therefore
\begin{equation}\label{eq:raw}
 s h_{d+1}(b)
 =M_d(x)+\sum_{v\in \F_q^*}\psi(v)
 h_d\left(\frac{x}{v^{2^{d-1}}}\right).
\end{equation}
The Frobenius map $v\mapsto v^{2^{d-1}}$ permutes $\F_q^*$, and
$\Tr(v^{2^{d-1}})=\Tr(v)$. After this change of variables, the sum in
\eqref{eq:raw} becomes
\[
 \sum_{w\in\F_q^*}\psi(w)h_d(x/w).
\]
We now distinguish two cases. If $x=0$, then
\[
 \sum_{w\in \F_q^*}\psi(w)h_d(0)=-h_d(0),
\]
and hence, by the first formula in \eqref{eq:M},
\[
 s h_{d+1}(1)
 =M_d(0)-h_d(0)
 =\frac{q-s}{s-1}h_d(0).
\]
Since
\begin{equation}\label{eq:B-ratio}
 \frac{B_{d+1}}{B_d}=\frac{q-s}{s(s-1)},
\end{equation}
division by $sB_{d+1}$ gives
\[
 p_{d+1}(1)=p_d(0),
\]
which is \eqref{eq:transfer-one}.

Suppose next that $x\ne0$. By the substitution $t=x/w$, one has
\[
\sum_{w\ne0}\psi(w)h_d(x/w)
=\sum_{t\ne0}\psi(x/t)h_d(t).
\]
For $x\ne0$, the second formula in \eqref{eq:M} gives
\[
 M_d(x)=\frac{B_d-h_d(0)}{s-1}.
\]
Substituting these two identities into \eqref{eq:raw}, using
$h_d(t)=B_dp_d(t)$, and then dividing by $B_{d+1}$ with
\eqref{eq:B-ratio}, we obtain
\[
 p_{d+1}(b)
 =\frac{1-p_d(0)+(s-1)\displaystyle\sum_{t\ne0}
 \psi(x/t)p_d(t)}{q-s}.
\]
Finally, take $x=b+1$, and this is exactly \eqref{eq:transfer-other}.
\end{proof}

For a real-valued function $f$ on $\F_q$, define
\begin{align}
 (\mathcal T_sf)(1)&=f(0),\label{eq:Ts-one}\\
 (\mathcal T_sf)(b)&=
 \frac1{q-s}\sum_{v\in\F_q^*}
 \bigl(1+(s-1)\psi((b+1)/v)\bigr)f(v),
 \qquad b\ne1.\label{eq:Ts-other}
\end{align}
The formulas are linear in $f$, so $\mathcal T_s$ is a linear operator on
the space of real-valued functions on $\F_q$. With this notation,
Theorem~\ref{thm:transfer} can be written as
\begin{equation}\label{eq:p-transfer}
 p_{d+1}=\mathcal T_{2^d}p_d.
\end{equation}

The operator $\mathcal T_s$ preserves the sum of the coordinates. Indeed,
for every $v\in\F_q^*$,
\[
 \sum_{1\ne b\in \F_q}\psi((b+1)/v)=-1,
\]
and hence a direct summation in
\eqref{eq:Ts-one}--\eqref{eq:Ts-other} gives
\[
 \sum_{b\in\F_q}(\mathcal T_sf)(b)
 =\sum_{b\in\F_q}f(b).
\]
Moreover, $\mathcal T_s$ fixes the uniform distribution
$u=q^{-1}\one$, that is, $\mathcal T_su=u$.
Indeed, $(\mathcal T_su)(1)=1/q$, while for $b\ne1$,
\[
 (\mathcal T_su)(b)
 =\frac{(q-1)-(s-1)}{q(q-s)}
 =\frac1q.
\]
Although some matrix entries of $\mathcal T_s$ may be negative, this is
irrelevant for what follows, since we use $\mathcal T_s$ only as a linear
operator.

\section{Fourier analysis of the transfer operator}\label{sec:fourier}

We now study the linear operator introduced at the end of the preceding
section. Recall that, for $1\le d<n$, we put $s=2^d$ and have
$p_{d+1}=\mathcal T_s p_d$,
where $\mathcal T_s$ is the operator defined by
\eqref{eq:Ts-one}--\eqref{eq:Ts-other}. We also recall that
$u=q^{-1}\one$ is the uniform distribution on $\F_q$ and that
$\mathcal T_su=u$. Hence
\[
 p_{d+1}-u=\mathcal T_s(p_d-u).
\]
Since
\[
 \sum_{x\in\F_q}(p_d(x)-u(x))=0,
\]
it is natural to study $\mathcal T_s$ on the subspace of functions with zero
coordinate sum.

Let $\mathbb R^{\F_q}$ denote the real vector space of all functions
$f:\F_q\to\mathbb R$, and put
\[
 \mathcal H=
 \left\{f\in\mathbb R^{\F_q}:\sum_{x\in\F_q}f(x)=0\right\}.
\]
Since $\mathcal T_s$ preserves the sum of the coordinates, $\mathcal H$ is
invariant under $\mathcal T_s$; thus $\mathcal T_s$ restricts to a linear
operator $\mathcal H\to\mathcal H$.

We equip $\mathbb R^{\F_q}$ with the standard Euclidean inner product and
norm
\[
 \langle f,g\rangle=\sum_{x\in\F_q}f(x)g(x),
 \qquad
 \|f\|_2=\left(\sum_{x\in\F_q}|f(x)|^2\right)^{1/2}.
\]
For $f\in\mathbb R^{\F_q}$, we write
\[
 f^\perp=\{g\in\mathbb R^{\F_q}:\langle f,g\rangle=0\}
\]
for its orthogonal complement, and for a linear operator
$A:\mathcal H\to\mathcal H$, we write
\[
 \|A\|_{\mathcal H\to\mathcal H}
 =\sup_{0\ne f\in\mathcal H}\frac{\|Af\|_2}{\|f\|_2}
\]
for the operator norm induced by the Euclidean norm.

For $a\in\F_q$, recall that $\delta_a$ denotes the delta function at $a$, and
define
\[
 w_a=\frac{\delta_a-u}{\sqrt{1-1/q}}.
\]
Since $\delta_a-u\in\mathcal H$ and
\[
 \|\delta_a-u\|_2^2=1-\frac1q,
\]
each $w_a$ is a unit vector in $\mathcal H$.

Finally, for later use, define
\[
 \alpha_s=\frac{\sqrt q\,(s-1)}{q-s}.
\]

\begin{lem}\label{lem:norm}
Let $1\le d<n$ and put $s=2^d$. For every $f\in\mathcal H$, we have
\begin{equation}\label{eq:norm}
 \|\mathcal T_sf\|_2^2
 =\alpha_s^2\|f\|_2^2
 +(1-\alpha_s^2)|\langle f,w_0\rangle|^2.
\end{equation}
Moreover,
\begin{equation}\label{eq:w0w1map}
 \mathcal T_sw_0=w_1.
\end{equation}
If $f\in w_0^\perp\cap\mathcal H$, then
\[
 \mathcal T_sf\in w_1^\perp\cap\mathcal H
 \qquad\text{and}\qquad
 \|\mathcal T_sf\|_2=\alpha_s\|f\|_2.
\]
More generally, for any $f,g\in w_0^\perp\cap\mathcal H$,
\[
 \langle\mathcal T_sf,\mathcal T_sg\rangle
 =\alpha_s^2\langle f,g\rangle.
\]
Consequently, $\alpha_s^{-1}\mathcal T_s$ maps
$w_0^\perp\cap\mathcal H$ isometrically onto
$w_1^\perp\cap\mathcal H$. In particular, if $s^2\le q$, then
\[
 \|\mathcal T_s\|_{\mathcal H\to\mathcal H}\le1.
\]
\end{lem}

\begin{proof}
Set $a=f(0)$ and define
\[
 F(x)=\sum_{v\in\F_q^*}\psi(x/v)f(v).
\]
Since $f\in\mathcal H$, we have
\[
 \sum_{x\in\F_q}f(x)=0.
\]
By additive character orthogonality,
\begin{equation}\label{eq:Fmom}
 \sum_{x\ne0}F(x)=a,
 \qquad
 \sum_{x\ne0}F(x)^2
 =q\bigl(\|f\|_2^2-a^2\bigr)-a^2.
\end{equation}
Indeed, for the first identity,
\[
 \sum_{x\ne0}\psi(x/v)=-1
 \qquad (v\ne0),
\]
while the second follows by expanding the square and using
\[
 \sum_{x\in\F_q}\psi\bigl(x(v^{-1}+w^{-1})\bigr)
 =
 \begin{cases}
 q,&v=w,\\
 0,&v\ne w.
 \end{cases}
\]

For $b\ne1$, equations \eqref{eq:Ts-one} and \eqref{eq:Ts-other} give
\[
 (\mathcal T_sf)(b)
 =\frac{-a+(s-1)F(b+1)}{q-s},
\]
whereas
\[
 (\mathcal T_sf)(1)=a.
\]
Expanding the square and applying \eqref{eq:Fmom}, we obtain
\[
 \|\mathcal T_sf\|_2^2
 =
 \frac{q(s-1)^2}{(q-s)^2}\|f\|_2^2
 +\frac{q(q-s^2)}{(q-s)^2}a^2.
\]
Since
\[
 |\langle f,w_0\rangle|^2
 =\frac{q}{q-1}a^2
\]
and
\[
 1-\alpha_s^2
 =\frac{(q-1)(q-s^2)}{(q-s)^2},
\]
formula \eqref{eq:norm} follows.

Next, since $\mathcal T_su=u$ and $\mathcal T_s\delta_0=\delta_1$, we have
\[
 \mathcal T_s(\delta_0-u)=\delta_1-u,
\]
and hence
\[
 \mathcal T_sw_0=w_1.
\]
This proves \eqref{eq:w0w1map}.

Now let $f\in w_0^\perp\cap\mathcal H$. Then
\[
 \langle f,w_0\rangle=0.
\]
Since
\[
w_0=\frac{\delta_0-u}{\sqrt{1-1/q}}
\]
and $f\in\mathcal H$, it follows that $f(0)=0$. Therefore
\[
 (\mathcal T_sf)(1)=0.
\]
Moreover, $\mathcal T_s$ preserves the coordinate sum, so
$\mathcal T_sf\in\mathcal H$. Since
\[
w_1=\frac{\delta_1-u}{\sqrt{1-1/q}},
\]
we obtain
\[
 \langle \mathcal T_sf,w_1\rangle=0.
\]
Thus
\[
 \mathcal T_sf\in w_1^\perp\cap\mathcal H.
\]

For such an $f$, formula \eqref{eq:norm} reduces to
\[
\|\mathcal T_sf\|_2=\alpha_s\|f\|_2.
\]
More generally, if $f,g\in w_0^\perp\cap\mathcal H$, then applying the
polarization identity to \eqref{eq:norm} gives
\[
\langle\mathcal T_sf,\mathcal T_sg\rangle
=\alpha_s^2\langle f,g\rangle.
\]
Hence
\[
\alpha_s^{-1}\mathcal T_s:
w_0^\perp\cap\mathcal H
\longrightarrow
w_1^\perp\cap\mathcal H
\]
is an isometry onto its image. The domain and codomain have the same
dimension, namely $q-2$, and $\alpha_s>0$ for $1<s<q$. Therefore the map is
onto.

Finally, if $s^2\le q$, then
\[
\alpha_s=\frac{\sqrt q\,(s-1)}{q-s}\le1.
\]
Together with \eqref{eq:norm}, this yields
\[
\|\mathcal T_sf\|_2\le\|f\|_2
\qquad(f\in\mathcal H),
\]
and hence
\[
\|\mathcal T_s\|_{\mathcal H\to\mathcal H}\le1.
\]
Lemma~\ref{lem:norm} is proved.
\end{proof}

Lemma~\ref{lem:norm} shows that, on the orthogonal complement of $w_0$,
the operator $\mathcal T_s$ multiplies Euclidean norms by the factor
$\alpha_s$. Thus, when $\alpha_s<1$, every such direction is contracted,
whereas the exceptional direction $w_0$ is mapped isometrically to $w_1$
by \eqref{eq:w0w1map}. The two unit vectors satisfy
\[
 \langle w_0,w_1\rangle=-\frac1{q-1},
\]
so they are nearly orthogonal when $q$ is large. This observation is the key
to the two-step estimate below.

\begin{lem}\label{lem:two-exact}
Let $s=2^d$ with $d\ge1$, and assume $s^2\le q/4$. Put
\[
a=\alpha_s,\quad b=\alpha_{2s},\quad c=\frac1{q-1},\quad
\Sigma_s=a^2+b^2+(1-a^2)(1-b^2)c^2.
\]
Then
\begin{equation}\label{eq:rho}
\|\mathcal T_{2s}\mathcal T_s\|_{\mathcal H\to\mathcal H}
 =
 \left(
\frac{\Sigma_s+\sqrt{\Sigma_s^2-4a^2b^2}}{2}
\right)^{1/2}
 =:\rho_s.
\end{equation}
If, in addition, $s^2\le q/32$, then
\begin{equation}\label{eq:two-contract}
 \rho_s<\frac{2s}{\sqrt q}.
\end{equation}
\end{lem}

\begin{proof}
We first make the orthogonal factorization of $\mathcal T_s$ explicit.
Using the orthogonal decomposition
\[
 \mathcal H=\operatorname{span}\{w_0\}
 \oplus (w_0^\perp\cap\mathcal H),
\]
define a linear operator $D_a:\mathcal H\to\mathcal H$ by
\[
 D_aw_0=w_0,
 \qquad
 D_af=af\quad(f\in w_0^\perp\cap\mathcal H).
\]
Formula \eqref{eq:norm} says precisely that
\[
 \|\mathcal T_sf\|_2^2=\|D_af\|_2^2
 \qquad(f\in\mathcal H).
\]
Let $\mathcal T_s^*$ denote the adjoint of $\mathcal T_s$ with respect to
the Euclidean inner product on $\mathcal H$, so that
\[
 \langle\mathcal T_sf,g\rangle
 =\langle f,\mathcal T_s^*g\rangle
 \qquad(f,g\in\mathcal H).
\]
By the polarization identity, the equality of norms above implies
\[
 \langle\mathcal T_sf,\mathcal T_sg\rangle
 =\langle D_af,D_ag\rangle
 \qquad(f,g\in\mathcal H).
\]
From the definition of $D_a$, we see that $D_a$ is self-adjoint. Hence
\[
 \langle f,\mathcal T_s^*\mathcal T_sg\rangle
 =\langle f,D_a^2g\rangle
 \qquad(f,g\in\mathcal H),
\]
and therefore we have $\mathcal T_s^*\mathcal T_s=D_a^2$

Since $a>0$, the operator $D_a$ is invertible. Define
$R_s:=\mathcal T_sD_a^{-1}$.
Since $D_a^{-1}$ is also self-adjoint, we have
\[
 R_s^*R_s
=D_a^{-1}\mathcal T_s^*\mathcal T_sD_a^{-1}
=D_a^{-1}D_a^2D_a^{-1}
 =I.
\]
Thus $R_s$ is orthogonal on $\mathcal H$, and by definition
$\mathcal T_s=R_sD_a$.
Moreover, since $D_aw_0=w_0$ and
$\mathcal T_sw_0=w_1$ by \eqref{eq:w0w1map}, we have
$ R_sw_0=w_1$

Define $D_b$ in the same way, with $b=\alpha_{2s}$ in place of $a$, and put
$R_{2s}:=\mathcal T_{2s}D_b^{-1}$.
The same argument gives
$T_{2s}=R_{2s}D_b$, where $R_{2s}$ is orthogonal. Consequently,
\[
 \mathcal T_{2s}\mathcal T_s
 =R_{2s}D_bR_sD_a.
\]
Since left multiplication by an orthogonal operator does not change
singular values, $\mathcal T_{2s}\mathcal T_s$ and
$D_bR_sD_a$ have the same singular values. We therefore   study
$A:=D_bR_sD_a$.

Now we put $z=R_s^{-1}w_0$. Since $R_s$ is orthogonal and $R_sw_0=w_1$, we have
\[
 \langle w_0,z\rangle
 =\langle R_sw_0,w_0\rangle
 =\langle w_1,w_0\rangle
 =-c.
\]
Let $E=\operatorname{span}\{w_0,z\}
$, $F=\operatorname{span}\{w_0,w_1\}$. We write
\[
 \mathcal H=E\oplus E^\perp,
 \qquad
 \mathcal H=F\oplus F^\perp
\]
for the orthogonal decompositions. Since
$R_sw_0=w_1$, $R_sz=w_0$,
we have
$R_sE=F$.
As $R_s$ is orthogonal, it also follows that
$R_sE^\perp=F^\perp$.
Note that the operators $D_a$ and $D_b$ preserve the corresponding decompositions, that is,
$D_a(E)=E$, $D_a(E^\perp)=E^\perp$; and 
$D_b(F)=F$, $D_b(F^\perp)=F^\perp$.
It follows that
\[
A(E)=F,\qquad A(E^\perp)=F^\perp.
\]
Thus, with respect to orthonormal bases adapted to the two orthogonal
decompositions, the matrix of $A$ is block diagonal. Consequently, the singular values
of $A$ on $\mathcal H$ are precisely the singular values of the two
restrictions
$A:E\to F$ and 
$ A:E^\perp\to F^\perp$
taken together.

We first consider the restriction of $A$ to $E^\perp$. If $f\in E^\perp$, then
$f\perp w_0$, so $D_af=af$.
Moreover, $R_sf\in F^\perp\subseteq w_0^\perp$, and hence
\[ D_bR_sf=bR_sf.
\]
Therefore
\[
 Af=abR_sf
 \qquad(f\in E^\perp).
\]
Since $R_s$ is orthogonal, its restriction
$R_s:E^\perp\to F^\perp$
is an isometry. Hence $A|_{E^\perp}$ is $ab$ times an isometry, and
every singular value of $A:E^\perp\longrightarrow F^\perp$
is equal to $ab$.

It remains to study the two-dimensional restriction
$A:E\longrightarrow F$. Put
$h=(1-c^2)^{1/2}$.
Since
$\langle w_0,z\rangle=-c$,
the vectors
\[
 e_1:=w_0,
 \qquad
 e_2:=h^{-1}(z+c w_0)
\]
form an orthonormal basis of $E$. Similarly,
\[
 f_1:=w_0,
 \qquad
 f_2:=h^{-1}(w_1+c w_0)
\]
form an orthonormal basis of $F$. Then we have
$w_1=-cf_1+hf_2$.
Moreover,
\[
 R_se_2
 =\frac{w_0+c w_1}{h}
 =hf_1+cf_2.
\]
Since $D_ae_1=e_1$ and $D_ae_2=ae_2$, while
$D_bf_1=f_1$ and $D_bf_2=bf_2$, we obtain
\[
 Ae_1=-cf_1+bhf_2
\]
and
\[
 Ae_2=ahf_1+abcf_2.
\]
Thus, with respect to the orthonormal input basis $(e_1,e_2)$ of $E$
and the orthonormal output basis $(f_1,f_2)$ of $F$, the matrix of this
two-dimensional restriction is
\[
 M=
 \begin{pmatrix}
 -c&ah\\
 bh&abc
 \end{pmatrix}.
\]
Let $\sigma_1,\sigma_2$ be the singular values of $M$. Since both bases
are orthonormal, these are exactly the singular values of the restriction
$A:E\to F$. Their squares are the eigenvalues of $M^{\mathsf T}M$.
Hence
\[
 \sigma_1^2+\sigma_2^2
 =\operatorname{tr}(M^{\mathsf T}M),
\]
where $\text{tr}$ denotes the trace of a matrix. Using $h^2=1-c^2$, we obtain
\[
\begin{aligned}
 \operatorname{tr}(M^{\mathsf T}M)
 &=c^2+a^2h^2+b^2h^2+a^2b^2c^2\\
 &=a^2+b^2+(1-a^2)(1-b^2)c^2\\
 &=\Sigma_s.
\end{aligned}
\]
Also, $\det M
 =-ab(c^2+h^2)
 =-ab$, and therefore
\[
\sigma_1^2\sigma_2^2
=\det(M^{\mathsf T}M)
 =(\det M)^2
 =a^2b^2.
\]
Thus $\sigma_1^2$ and $\sigma_2^2$ are the two roots of
$X^2-\Sigma_sX+a^2b^2=0$. Let $\lambda_+$ denote the larger root, i.e.,
\[
 \lambda_+
 =\frac{\Sigma_s+\sqrt{\Sigma_s^2-4a^2b^2}}{2}.
\]

We now compare this two-dimensional block with the restriction to
$E^\perp$. Under the assumption $s^2\le q/4$, we have $0<a,b\le1$. Since the product of the two roots is $a^2b^2$, their larger root satisfies
$\lambda_+\ge ab$.
Moreover, $ab\ge a^2b^2$.
On $E^\perp$, all singular values of $A$ are equal to $ab$, so their
squares are $a^2b^2$. Hence
$\lambda_+$ is the largest squared singular value of $A$ on $\mathcal{H}$. Therefore
\[
 \|A\|_{\mathcal H\to\mathcal H}^2
 =\lambda_+
 =\frac{\Sigma_s+\sqrt{\Sigma_s^2-4a^2b^2}}{2}.
\]
Since $R_{2s}$ is orthogonal, we conclude that
\[
 \|\mathcal T_{2s}\mathcal T_s\|_{\mathcal H\to\mathcal H}
 =
 \left(
 \frac{\Sigma_s+\sqrt{\Sigma_s^2-4a^2b^2}}{2}
 \right)^{1/2}
 =\rho_s.
\]

We next prove \eqref{eq:two-contract}. Assume that
$s^2\le q/32$, or equivalently $q\ge32s^2$. Recall that, on $E^\perp$, every singular
value of $A=D_bR_sD_a$ is equal to $ab$, while on $E$ the matrix of $A$
with respect to the orthonormal bases chosen above is
\[
 M=
 \begin{pmatrix}
 -c&ah\\
 bh&abc
 \end{pmatrix},
 \qquad h=(1-c^2)^{1/2}.
\]
We estimate the norm of this two-dimensional block directly. Write
\[
 M=
 \begin{pmatrix}
 0&ah\\
 bh&0
 \end{pmatrix}
 +
 \begin{pmatrix}
 -c&0\\
 0&abc
 \end{pmatrix}
 =:M_0+M_1.
\]
Since $b>a>0$,
the singular values of $M_0$ are $ah$ and $bh$, and hence
$ \|M_0\|=bh\le b$.
Also, since $0<a,b\le1$, $\|M_1\|=\max\{c,abc\}=c$.
It follows from the triangle inequality for the operator norm that
\[
 \|M\|\le b+c.
\]
The singular value on $E^\perp$ is $ab\le b$, so the block decomposition
of $A$ gives
\[
 \rho_s=\|A\|_{\mathcal H\to\mathcal H}\le b+c.
\]
In the following, let us compare $b+c$ with  $2s/\sqrt q$. Note that
\[
\begin{aligned}
 \frac{2s}{\sqrt q}-b
 =\frac{2s(q-2s)-q(2s-1)}{\sqrt q\,(q-2s)}=\frac{q-4s^2}{\sqrt q\,(q-2s)}.
\end{aligned}
\]
The assumption $q\ge32s^2$ gives $q-4s^2\ge\frac78q$,
and therefore
\begin{align}\label{b:1}
\frac{2s}{\sqrt q}-b
 >\frac{7}{8\sqrt q}.
\end{align}
Moreover, $s\ge2$ implies $q\ge128$, and hence
\begin{align}\label{c:1}
 c=\frac1{q-1}<\frac1{8\sqrt q}.
\end{align}
Putting \eqref{b:1} and \eqref{c:1} together, we have
\[
\frac{2s}{\sqrt q}-b-c>\frac3{4\sqrt q}>0,
\]
and thus $b+c<\frac{2s}{\sqrt q}$. Therefore, we obtain $\rho_s<\frac{2s}{\sqrt q}$, which proves \eqref{eq:two-contract}.
\end{proof}

\begin{rem}
The exact expression \eqref{eq:rho} is stronger than what is needed for the
asymptotic estimate. Its value here is that it makes the two-step argument
transparent: the direction that is not contracted in the first step is almost
orthogonal to the direction that can avoid contraction in the second.
\end{rem}

\section{Proof of the main results}\label{sec:proof}

\begin{proof}[Proof of Theorem~\ref{thm:quant}]
Set
\[
 f_d=p_d-u.
\]
Since $\mathcal T_su=u$, equation \eqref{eq:p-transfer} gives
\[
 f_{d+1}
 =p_{d+1}-u
 =\mathcal T_{2^d}(p_d-u)
 =\mathcal T_{2^d}f_d.
\]
Moreover, $p_1=\delta_1$, and hence
\[
 \|f_1\|_2
 =\|\delta_1-u\|_2
 =\sqrt{1-\frac1q}
 <1.
\]

There are $k-1$ transfers from $f_1$ to $f_k$. Put
$r=\lfloor\frac{k-1}{2}\rfloor$.
We first group the initial $2r$ transfers into $r$ consecutive pairs,
whose starting indices are
$d=1,3,\ldots,2r-1$.
For the pair beginning at $d=2j+1$, where $0\le j\le r-1$, put
\[
 s=2^d=2^{2j+1}.
\]
Since $d\le k-2$, the assumption $n\ge2k+1$ gives
$n\ge2d+5$.
Therefore
\[
q=2^n\ge2^{2d+5}=32s^2,
\]
and Lemma~\ref{lem:two-exact} yields
\[
\begin{aligned}
\|f_{2j+3}\|_2
 = \|\mathcal T_{2^{2j+2}}\mathcal T_{2^{2j+1}} f_{2j+1}\|_2<
\frac{2\cdot2^{2j+1}}{\sqrt q}\,
 \|f_{2j+1}\|_2.
\end{aligned}
\]
Iterating this estimate for $j=0,\ldots,r-1$, we obtain
\[
\begin{aligned}
 \|f_{2r+1}\|_2
 <
 \left(
 \prod_{j=0}^{r-1}
 \frac{2\cdot2^{2j+1}}{\sqrt q}
 \right)\|f_1\|_2<
\prod_{j=0}^{r-1}
\frac{2\cdot2^{2j+1}}{\sqrt q}=
2^{r^2+r}q^{-r/2},
\end{aligned}
\]
where in the equality we used
$\sum_{j=0}^{r-1}(2j+2)=r^2+r$.

If $k-1=2r$, then $k=2r+1$, so $f_k=f_{2r+1}$ and hence
\[
 \|f_k\|_2<2^{r^2+r}q^{-r/2}.
\]
If $k-1=2r+1$, then one transfer remains and
\[
f_k=\mathcal T_{2^{k-1}}f_{2r+1}.
\]
Since
$2^{2k-2}<2^n=q$, by
Lemma~\ref{lem:norm}, one has
\[
\|\mathcal T_{2^{k-1}}\|_{\mathcal H\to\mathcal H}\le1.
\]
Therefore
\[
 \|f_k\|_2
\le\|\mathcal T_{2^{k-1}}\|_{\mathcal H\to\mathcal H}\cdot\|f_{2r+1}\|_2
<2^{r^2+r}q^{-r/2}.
\]
Thus, in both cases,
\begin{equation}\label{eq:fk-bound}
 \|p_k-u\|_2=\|f_k\|_2
 <2^{r^2+r}q^{-r/2}.
\end{equation}

Note from Lemma~\ref{lem:pzero} that $p_k(0)=\frac{N_{n,k}}{G_{n,k}}$,
whereas $u(0)=1/q$. Since the absolute value of any coordinate of a
vector is at most its Euclidean norm, \eqref{eq:fk-bound} gives
\[
 \left|
 \frac{N_{n,k}}{G_{n,k}}-\frac1q
 \right|
 \le
 \|p_k-u\|_2
 <2^{r^2+r}q^{-r/2}.
\]
Multiplying by $q=2^n$, we obtain
\[
 \left|
 \frac{2^nN_{n,k}}{G_{n,k}}-1
 \right|
 <
 2^{r^2+r}2^{n(1-r/2)},
\]
which proves \eqref{eq:mainbound}.

We next prove \eqref{eq:positive-half}. Assume $k\ge10$. Then $r\ge4$,
and the exponent
\[
 E(n,r):=r^2+r+n\left(1-\frac r2\right)
\]
in the right-hand side of \eqref{eq:mainbound} is decreasing as a
function of $n$. Hence it is enough to evaluate it at the smallest
allowed value $n=2k+1$.

If $k=2r+1$ is odd, then $n=4r+3$, and
\[
 E(4r+3,r)
 =-r^2+\frac72r+3.
\]
If $k=2r+2$ is even, then $n=4r+5$, and
\[
 E(4r+5,r)
 =-r^2+\frac52r+5.
\]
For $k\ge10$, both quantities are at most $-1$, and equality occurs only
when $(k,n)=(10,21)$. Since the two-step estimate
\eqref{eq:two-contract} is strict, \eqref{eq:mainbound} gives
\[
 \left|
 \frac{2^nN_{n,k}}{G_{n,k}}-1
 \right|
 <\frac12.
\]
Therefore
\[
 \frac{2^nN_{n,k}}{G_{n,k}}>\frac12,
\]
and hence
\[
 N_{n,k}>\frac{G_{n,k}}{2^{n+1}}.
\]
This proves \eqref{eq:positive-half}.

Finally, assume that
$\frac n3<k\le\frac{n-1}{2}$,
and put $m=\lfloor\frac n6\rfloor$. Then we have 
\[
 m\le r\le\frac n4
\]
since $r=\lfloor\frac{k-1}{2}\rfloor$. Let
\[
F_n(r):=n+r^2+r-\frac{nr}{2}
\]
be a function of $r$. Clearly,  $F_n$ is convex. Hence its maximum on
$[m,n/4]$ is attained at one of the two endpoints.

We may write $m=\frac n6-\delta$ for some $\delta$ with $0\le\delta<1$.
At the left endpoint, we compute that
\[
\begin{aligned}
 F_n(m)
 =
 -\frac{n^2}{18}
 +\left(\frac76+\frac{\delta}{6}\right)n
 +\delta^2-\delta\le
 -\frac{n^2}{18}+\frac{4n}{3}.
\end{aligned}
\]
At the right endpoint,
\[
 F_n(n/4)
 =-\frac{n^2}{16}+\frac{5n}{4}
 <-\frac{n^2}{18}+\frac{4n}{3}.
\]
Consequently,
\[
 F_n(r)
 <-\frac{n^2}{18}+\frac{4n}{3}
\]
throughout the range $n/3<k\le(n-1)/2$. Substituting this into
\eqref{eq:mainbound} gives
\[
 \sup_{n/3<k\le(n-1)/2}
 \left|
 \frac{2^nN_{n,k}}{G_{n,k}}-1
 \right|
 <
 2^{-n^2/18+4n/3},
\]
which proves \eqref{eq:uniformcentral}.
\end{proof}

\begin{proof}[Proof of Theorem~\ref{thm:carlet}]
By \eqref{eq:complement}, we may replace $k$ by $\min(k,n-k)$ and assume
$k\le n/2$. If $n$ is even and $k=n/2$, the subfield
$\F_{2^k}\subseteq\F_{2^n}$ is a zero-sum $k$-space, since inversion
permutes its nonzero elements and the sum of all elements of $\F_{2^k}$ is
zero.

In every remaining case we have $n\ge2k+1$. For $k\ge10$,
Theorem~\ref{thm:quant} implies $N_{n,k}>0$. It remains only to consider
$3\le k\le9$. Hou and Zhao \cite[Corollary~5.5]{HouZhaoFFA} proved the
required conclusion for every $3\le k\le12$ under the condition
$n\ge k+3$, which is automatic here. Thus $N_{n,k}>0$ throughout the full
range $3\le k\le n-3$. This proves the conjecture.
\end{proof}

\begin{rem}
The recursion of Theorem~\ref{thm:transfer} also provides an exact procedure
for computing the numbers $N_{n,k}$ when $n$ is fixed. Formula
\eqref{eq:raw} may be evaluated by an additive Fourier transform after the
nonzero field elements are permuted by inversion. This computational
observation is not needed for the proof.
\end{rem}

\section*{Acknowledgments}

The author is grateful to Domingo G\'omez-P\'erez for suggesting this problem
during the author's visit to the University of Cantabria in August 2025.

\end{document}